\documentclass[12pt,reqno]{amsart}
\usepackage{amssymb}
\usepackage{amsfonts}
\usepackage{hyperref}
\usepackage{graphicx}
\usepackage{color}
\usepackage{amsmath}
\usepackage[font=scriptsize,labelfont=bf]{caption}

\graphicspath{{./CurtisPictureFile/}}
\definecolor{myurlcolor}{rgb}{0,0.6,0.2}
\hypersetup{colorlinks, linkcolor=myurlcolor, citecolor=myurlcolor, urlcolor=myurlcolor}
\def\equationautorefname~#1\null{(#1)\null}

\theoremstyle{plain}

\newtheorem{algorithm}{Algorithm}[section]
\newtheorem{thm}{Thm}

\newtheorem{corollary}[algorithm]{Corollary}

\newtheorem{definition}[algorithm]{Definition}
\newtheorem*{Remark-delta}{Remark on all things $\protect\delta$}

\newtheorem{lemma}[algorithm]{Lemma}

\newtheorem{theorem} [algorithm] {Theorem}
\newtheorem{theoremlet}[thm]{Theorem}

\newtheorem{definitionlet}[thm]{Definition}

\newtheorem*{definitionnonum}{Definition}

\newtheorem{proposition}[algorithm]{Proposition}

\numberwithin{equation}{algorithm}

\begin{document}
\title{On Delaunay Triangulations of Gromov Sets}
\author{Curtis Pro}
\address{Department of Mathematics, California State University, Stanislaus}
\email{cpro@csustan.edu}
\author{Frederick Wilhelm}
\address{Department of Mathematics, University of California, Riverside}
\email{fred@math.ucr.edu}
\urladdr{https://sites.google.com/site/frederickhwilhelmjr/home}
\thanks{This work was supported from by a grant from the Simons Foundation
(\#358068)}

\begin{abstract}
Let $Y$ be a subset of a metric space $X.$ We say that $Y$ is $\eta $%
--Gromov provided $Y$ is $\eta $--separated and not properly contained in
any other $\eta $--separated subset of $X.$ In this paper, we review a
result of Chew which says that any $\eta $-Gromov subset of $\mathbb{R}^{2} $
admits a triangulation $\mathcal{T}$ whose smallest angle is at least $\pi
/6 $ and whose edges have length between $\eta $ and $2\eta .$ We then show
that given any $k = 1,2,3\ldots$, there is a subdivision $\mathcal{T}
_{k}$ of $\mathcal{T}$ whose edges have length in $\left[ \frac{\eta}{10 k}
 ,\frac{2\eta}{10 k} \right] $ and whose minimum angle is also $%
\pi /6$.

These results are used in the proof of the following theorem in [10]: For
any $k\in R,v>0,$ and $D>0,$ the class of closed Riemannian $4$--manifolds
with sectional curvature $\geq k,$ volume $\geq v,$ and diameter $\leq D$
contains at most finitely many diffeomorphism types. Additionally, these
results imply that for any $\varepsilon >0$, if $\eta >0$ is sufficiently
small, any $\eta $--Gromov subset of a compact Riemannian $2$--manifold
admits a geodesic triangulation $\mathcal{T}$ for which all side lengths are
in $\left[ \eta \left( 1-\varepsilon \right) ,2\eta \left( 1+\varepsilon
\right) \right] $ and all angles are $\geq \frac{\pi }{6}-\varepsilon .$
\end{abstract}

\maketitle

Let $\mathcal{M}_{k,v,d}^{K,V,D}\left( n\right) $ denote the class of closed
Riemannian $n$--manifolds $M$ with 
\begin{equation*}
\begin{array}{cclccc}
k & \leq & \sec \,M & \leq & K, &  \\ 
v & \leq & \mathrm{vol}\,M & \leq & V, & \mathrm{and} \\ 
d & \leq & \mathrm{diam}\,M & \leq & D, & 
\end{array}%
\end{equation*}%
where $\sec $ is sectional curvature, $\mathrm{vol}$ is volume, and $\mathrm{%
diam}$ is diameter.

Cheeger's finiteness theorem says that $\mathcal{M}_{k,v,0}^{K,\infty
,D}\left( n\right) $ contains only finitely many diffeomorphism types (\cite%
{Cheeg1}, \cite{Cheeg2}, \cite{KirSie}, \cite{Pete}). In \cite{ProWilh}, we
prove the following, which generalizes Cheeger's finiteness theorem and the
result of Grove, Petersen and Wu from \cite{GrovPetWu}.

\begin{theoremlet}
\label{finiteness} For any $k\in \mathbb{R},$ $v>0,$ $D>0,$ and $n\in 
\mathbb{N},$ the class of closed Riemannian $n$--manifolds $\mathcal{M}%
_{k,v,0}^{\infty ,\infty ,D}\left( n\right) $ contains at most finitely many
diffeomorphism types.
\end{theoremlet}

Except in dimension $4,$ this result was established in the early 1990s via
work of Grove--Petersen--Wu, Perelman, and Kirby-Siebenmann in \cite%
{GrovPetWu}, \cite{Perel1}, \cite{Kap2}, and \cite{KirSie}. For further
details on finiteness theorems we refer the reader to \cite{Gro}. Our
argument in \cite{ProWilh} only treats the case $n=4,$ and depends on the
fact that there is a special family of simplicial complexes $\mathcal{T}$ in 
$\mathbb{R}^{2}.$ This family has a uniform lower bound for all angles and
certain subdivision and extension properties. The purpose of this paper is
to establish the existence of such a family. More specifically, we show that
there are nonempty examples of

\begin{definitionlet}
\label{SubEx dfn}Let $\mathcal{F}$ be a family of triangulations of subsets
of $\mathbb{R}^{2}.$ Given $\theta _{0},\sigma _{0}>0,$ we say that $%
\mathcal{F}$ is $\left( \theta _{0},\sigma _{0}\right) $--nondegenerate,
extendable, and subdividable provided:

\begin{enumerate}
\item All angles of all triangles in all $\mathcal{T\in F}$ are $\geq \theta
_{0}.$

\item For all $\mathcal{T\in F}$ all ratios of all edge lengths of $\mathcal{%
T}$ are bounded from above by $\sigma _{0}.$

\item For all $\mathcal{T\in F}$ there is a $\mathcal{T}_{\mathrm{ext}}%
\mathcal{\in F}$ so that $\mathcal{T}_{\mathrm{ext}}$ triangulates $\mathbb{R%
}^{2}$ and $\mathcal{T}\subset \mathcal{T}_{\mathrm{ext}}.$

\item Given any $\varepsilon >0$ and any $\mathcal{T\in F}$ there is a
subdivision $\mathcal{\tilde{T}\ }$of $\mathcal{T}$ so that $\mathcal{\tilde{%
T}\in F}$ and all edges of $\mathcal{\tilde{T}}$ have length $<\varepsilon .$
\end{enumerate}
\end{definitionlet}

Here we show

\begin{theoremlet}
\label{CDG them}The family $\mathcal{F}$ of $\left( \frac{\pi }{6},2\right) $%
--nondegenerate, extendable, and subdividable triangulations of $\mathbb{R}%
^{2}$ is not empty.
\end{theoremlet}

There are numerous papers in the computer science and computational geometry
literature that address the non-degeneracy and extension aspects of this
theorem (see e.g. \cite{ChengDeyShew} and the references therein). Among
these, Chew's is the most useful for our purposes (\cite{Chew}). It is based
on Delaunay triangulations of what we have decided to call Gromov sets.

A triangulation $\mathcal{T}$ of a point set of $\mathbb{R}^{2}$ is called 
\textbf{Delaunay} if and only if the circumdisk of each $2$--simplex
contains no vertices of $\mathcal{T}$ in its interior (see e.g. \cite{GaHoff}%
, Chapter 6). Among all possible triangulations of a given point set, it is
well known that a Delaunay triangulation maximizes the minimal angle (see
e.g. Theorem 9.9 of \cite{de berg}).

\begin{figure}[tbp]
\includegraphics[scale=.34]{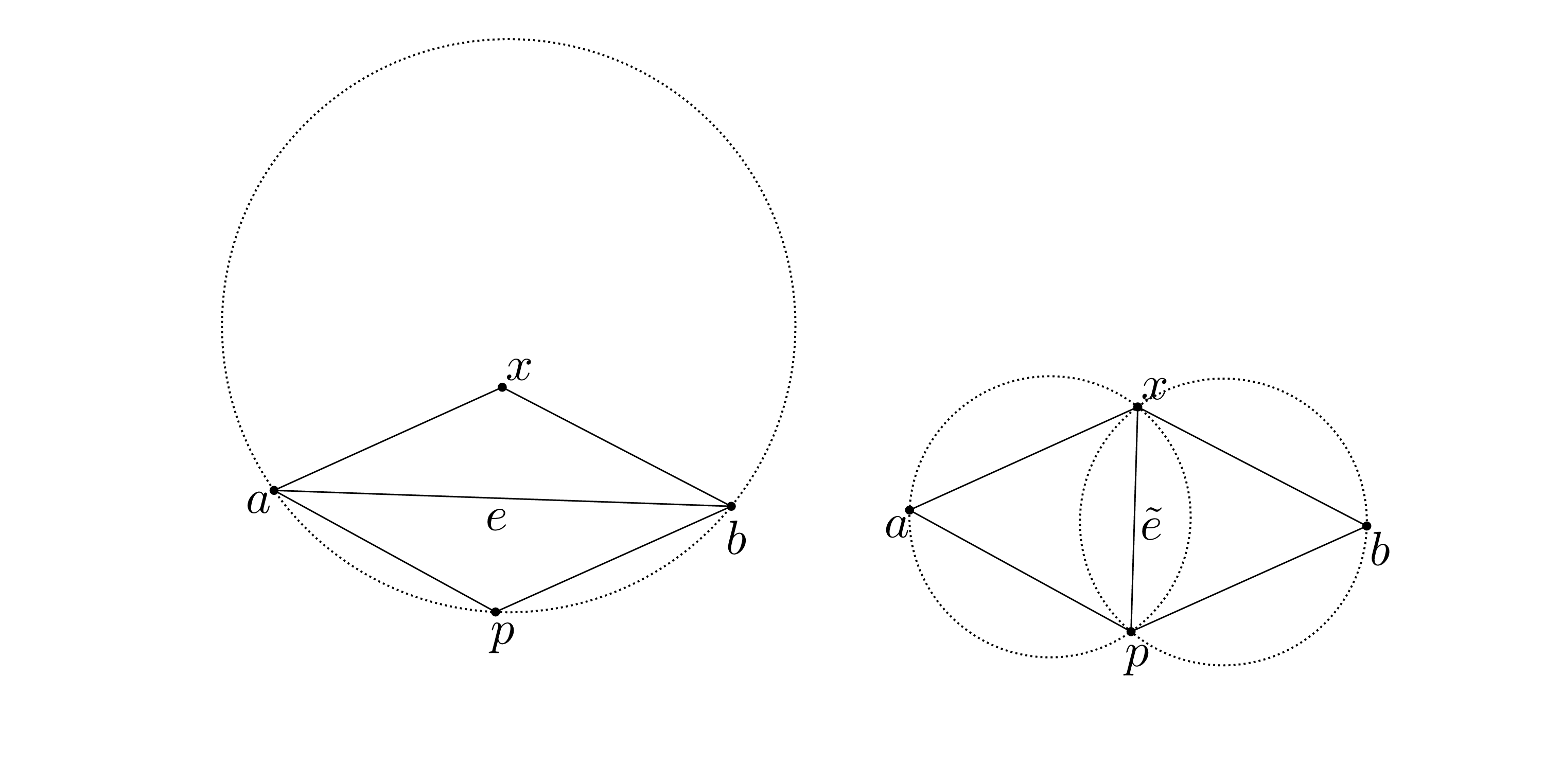}
\caption{{\protect\tiny {The triangulation of the four vertices on the left
fails the circumdisk property and is not Delaunay. The opposite is true for
the triangulation on the right.} }}
\label{lawlabel}
\end{figure}

The notion of a \textbf{Gromov subset }of a metric space is motivated by the
proof of Gromov's precompactness theorem, the notion of totally bounded\
metric spaces, and Gromov, Perelman, and Burago's notion of rough volume (%
\cite{BGP},\cite{Petersen}).

\begin{definitionnonum}
Given a metric space $X$ and $\eta >0,$ we say that $Y\subset X$ is $\eta $%
\textbf{--Gromov} provided $Y$ is $\eta $--separated and maximal with
respect to inclusion, that is, $Y$ is not properly contained in any $\eta $%
--separated subset of $X.$
\end{definitionnonum}

\begin{definitionnonum}[Chew-Delaunay-Gromov Complexes]
We call a simplicial complex $\mathcal{T}$ in $\mathbb{R}^{2}$ an $\eta $%
\textbf{--CDG}, provided, $\mathcal{T}$ is a subcomplex of a Delaunay
triangulation $\mathcal{\hat{T}}$ of an $\eta $\textbf{--}Gromov subset of $%
\mathbb{R}^{2}.$ If $\mathcal{T=\hat{T}},$ then $\mathcal{T}$ is called a 
\textbf{maximal CDG}$.$
\end{definitionnonum}

The algorithm on Pages 7--8 of \cite{Chew} combined with the Theorem on Page
9 and the Corollary on Page 10 of \cite{Chew} give the following.

\begin{theoremlet}[Chew's Angle Theorem]
\label{Gromov-Del Prop} Let $\mathcal{T}$ be a CDG complex in $\mathbb{R}%
^{2} $. Then all angles of $\mathcal{T}$ are $\geq \frac{\pi }{6}$ and all
edges of $\mathcal{T}$ are in $\left[ \eta ,2\eta \right] .$
\end{theoremlet}

Any non-colinear collection of points in $\mathbb{R}^{2}$ admits a Delaunay
triangulation (\cite{GaHoff}, Theorem 6.10), and, by definition, CDGs can
always be extended to all of $\mathbb{R}^{2}$, so Chew's Angle Theorem
implies that the family of all CDGs satisfies Properties 1--3 of Definition %
\ref{SubEx dfn} with $\left( \theta _{0},\sigma _{0}\right) =\left( \frac{%
\pi }{6},2\right) .$ Thus to prove Theorem \ref{CDG them}, it suffices to
show that CDGs also satisfy Property 4 of Definition \ref{SubEx dfn}.

The proof of Theorem \ref{finiteness} also exploits complexes that are close
to, but not quite, $\mathrm{CDGs.}$ In particular, we will study Delaunay
triangulations of sets $S_{i}$ with the same combinatorial structure as a
fixed Delaunay triangulation on a set $S,$ provided $S_{i}$ and $S$ are
sufficiently close. One complication is that a fixed subset $S\subset 
\mathbb{R}^{2}$ can have more than one Delaunay triangulation, hence subsets 
$S_{i}$ arbitrarily close to $S$ can have Delaunay triangulations
combinatorially distinct from a prescribed triangulation of $S.$ The
following two definitions are part of our strategy to account for these
issues.

\begin{definitionlet}
\label{del grph dfn}\emph{(\cite{GaHoff}, page 74)} For $S\subset \mathbb{R}%
^{2}$ we say a segment $e$ between two points of $S$ is in the \textbf{%
Delaunay graph} of $S$ if and only if $e$ is an edge of every Delaunay
triangulation of $S.$
\end{definitionlet}

\begin{definitionlet}
\label{stable edge dfn}Given a discrete $S\subset \mathbb{R}^{2},$ $\xi >0$,
and a segment $ab$ between two points $a$ and $b$ of $S,$ we say that $ab$
is \textbf{$\xi $--stable} provided the following holds. For any embedding $%
\iota : S \longrightarrow \mathbb{R}^{2}$ so that 
\begin{eqnarray*}
\left\vert \iota -\mathrm{id}_{S}\right\vert &<&\xi ,
\end{eqnarray*}%
the segment $\iota \left( a\right) \iota \left( b\right) $ is in the
Delaunay graph of $\iota \left( S\right) .$
\end{definitionlet}

The fact that $\mathrm{CDGs}$ satisfy Property 4 of Definition \ref{SubEx
dfn} is a consequence of the following result.

\begin{theoremlet}
\label{alm leg subdiv thm}Let $\mathcal{T}$ be an $\eta $--\textrm{CDG. }%
There is a subdivision $\mathcal{\tilde{T}}$ of $\mathcal{T}$ that is an $%
\frac{\eta }{10}$--\textrm{CDG}. Moreover, each edge $\tilde{e}$ of $%
\mathcal{\tilde{T}}$ that is a subedge of an edge $e$ of $\mathcal{T}$
satisfies one of the following conditions:

\begin{enumerate}
\item If $\tilde{e}$ does not contain a vertex of $e,$ then $e$ is $\left( 
\frac{1}{100}\frac{\eta }{10}\right) $--stable.

\item If $\tilde{e}$ contains a vertex of $e,$ then 
\begin{equation*}
\mathrm{length}\left( \tilde{e}\right) =\frac{\mathrm{length}\left( e\right) 
}{10}.
\end{equation*}%
In particular, each triangle $\tilde{\Delta}$ of $\mathcal{\tilde{T}}$ that
has a vertex in $\mathcal{T}$ is similar to the triangle $\Delta $ of $%
\mathcal{T}$ with $\tilde{\Delta}\subset \Delta .$
\end{enumerate}
\end{theoremlet}

In Section \ref{rev sect}, we review some basic facts about Delaunay
triangulations. In Section \ref{chew sect}, we review the proof of Theorem %
\ref{Gromov-Del Prop}. In Section \ref{alm legal sub div sect}, we prove
Theorem \ref{alm leg subdiv thm}, and in Section \ref{epilog section} we
explore various deformations of Theorems \ref{Gromov-Del Prop} and \ref{alm
leg subdiv thm} that we will need to prove Theorem \ref{finiteness}.
Throughout the paper, we let $\mathcal{T}_{k}$ denote the set of $k$%
--simplices of a simplicial complex $\mathcal{T}$. We let $\left\vert 
\mathcal{T}\right\vert $ be the polyhedron determined by $\mathcal{T}.$

\section{ Review of Delaunay Triangulations\label{rev sect}}

\begin{definition}
\emph{(\cite{GaHoff}, Definition 6.8)} The circumdisk of a triangle in $%
\mathbb{R}^{2}$ is the unique disk whose boundary circle passes through the
three vertices of the triangle (see Figure \ref{lawlabel}).

A triangulation $\mathcal{T}$ of a discrete point set $P\subset \mathbb{R}%
^{2}$ is called Delaunay if and only if every circumdisk of every triangle
in $\mathcal{T}$ contains no points of $P$ in its interior.
\end{definition}

The existence of Delaunay triangulations of discrete point sets in $\mathbb{R%
}^{2}$ is guaranteed by the following theorem.

\begin{theorem}
\label{Delauny exist and unique them}\emph{(\cite{GaHoff}, Theorem 6.10,
Lemma 6.16, Figure 6.6c)} Every discrete point set $V\subset \mathbb{R}^{2}$
has a Delaunay triangulation, provided $V$ does not lie on a line. The
Delaunay triangulation of $V$ is unique provided no four points of $V$ lie
on a circle. On the other hand, if four points of $V$ lie on a circle, then $%
V$ has more than one Delaunay triangulation.
\end{theorem}

The proof takes any triangulation $\mathcal{T}$ of $V$ and then performs a
sequence of edge replacements to $\mathcal{T}$ described as follows (see
Figure \ref{lawlabel}). If $\Delta _{1}=\Delta apb$ and $\Delta _{2}=\Delta
axb$ are two triangles of $\mathcal{T}$ that share a common edge $e=ab$, the
following test is applied to $e$ to determine if it should be replaced:

\noindent \textbf{Lawson Flip Test:} \emph{Let $D_{\Delta _{1}}$ be the
circumdisk of $\Delta _{1}.$ }

\begin{itemize}
\item If $x\in \mathrm{int}D_{\Delta _{1}},$ then replace $e=ab$ with $%
\tilde{e}=px$.

\item If $x\notin D_{\Delta _{1}},$ then do not replace $e=ab$.

\item If $x\in \partial D_{\Delta _{1}},$ then either $e=ab$ or $\tilde{e}%
=px $ are acceptable.
\end{itemize}

This algorithm produces a Delaunay triangulation of $\mathcal{T}$. In fact,

\begin{lemma}
\label{locally Del lemma}\emph{(See e.g., Proposition 6.13 in \cite{GaHoff}.)%
} Let $e$ be an edge of a triangulation $\mathcal{T}$ in $\mathbb{R}^{2}.$
If $e$ passes the Lawson flip test, then $e$ is an edge of a Delaunay
triangulation of $\mathcal{T}_{0}.$
\end{lemma}

For our purposes, it will be convenient to use an alternative to the Lawson
Flip Test which we call the\newline

\noindent \textbf{Angle Flip Test: }\emph{Let } $Q=apbx$\emph{\ be the
quadrilateral formed by }$\Delta _{1},$\emph{\ and }$\Delta _{2}$. \emph{\
Let $\sphericalangle a$, $\sphericalangle p$, $\sphericalangle b$, and $%
\sphericalangle x$, denote the angles of $Q$ at $a,p,b$, and $x$,
respectively and note} $\sphericalangle a+\sphericalangle p+\sphericalangle
b+\sphericalangle x=2\pi $.

\begin{itemize}
\item If $\sphericalangle p+\sphericalangle x>\pi $, then replace $e=ab$
with $\tilde{e}=px.$

\item If $\sphericalangle p+\sphericalangle x<\pi $, then do not replace $%
e=ab$.

\item If $\sphericalangle p+\sphericalangle x=\sphericalangle
a+\sphericalangle b=\pi $, then either $e=ab$ or $\tilde{e}=px$ is
acceptable.
\end{itemize}

To see that these tests are equivalent recall

\begin{theorem}
\emph{(Thale's Theorem, page 194, \cite{de berg})} Let $C$ be a circle in $%
\mathbb{R}^{2}$ that contains the points $a,p,b,$ and $q.$ Suppose that the
four points $b,q,$ $r,$ and $s,$ lie on the same side of $ap$ and that $s$
is outside of $C$ and that $r$ is inside of $C.$ Then 
\begin{equation*}
\sphericalangle \left( a,r,p\right) >\sphericalangle \left( a,b,p\right)
=\sphericalangle \left( a,q,p\right) >\sphericalangle \left( a,s,p\right) .
\end{equation*}
\end{theorem}

Let $Q=apbx$ be a quadrilateral in $\mathbb{R}^{2}$ whose vertices $a,p,b,x$
are listed in counterclockwise order. If $Q$ is a rectangle then all four
vertices lie on a circle $C$, and 
\begin{equation}
\sphericalangle a+\sphericalangle b=\sphericalangle p+\sphericalangle x=\pi .
\label{on circle}
\end{equation}

If $Q$ is perturbed in a way so that the points $a$ and $b$ are fixed and $%
x,p\in C$, then by Euclid's Central Angle Theorem, the angles $%
\sphericalangle x$ and $\sphericalangle p$ remain constant. So (\ref{on
circle}) holds if and only if $Q$ is inscribed in a circle $C$.

Again fix three points $a,p,$ and $b$ on $C.$ If $x$ is outside of $C,$ then
by Thale's theorem $\sphericalangle x$ is smaller than it is in (\ref{on
circle}), so $\sphericalangle p+\sphericalangle x<\pi $. Conversely, if $x$
is inside $C$, $\sphericalangle p+\sphericalangle x>\pi $. Thus we have
proven

\begin{proposition}
\label{two tests are one prop}Let $e$ be an edge of a triangulation $%
\mathcal{T}$ in $\mathbb{R}^{2}.$ $e$ passes the Lawson flip test if and
only if $e$ passes the angle flip test.
\end{proposition}

\begin{definition}
An edge $e$ of a triangulation $\mathcal{T}$ will be called a \textbf{%
Delaunay edge} provided $e$ is not replaced by the Angle Flip Test. A
triangle $\Delta $ in $\mathcal{T}$ is called a \textbf{Delaunay triangle}
provided each of its edges are Delaunay.
\end{definition}

\section{Chew's Angle Theorem \label{chew sect}}

To prove Theorem \ref{Gromov-Del Prop}, it suffices to consider the case
when $\mathcal{T}$ is a maximal CDG. This, together with the fact that all
edges of an $\eta $--\textrm{CDG }have length $\geq \eta $, means that
Theorem \ref{Gromov-Del Prop} follows from the following result.

\begin{theorem}
\label{jacked up Chew}Let $\mathcal{T}_{0}$ be a subset of $\mathbb{R}^{2}$
that is $\eta $--dense, and let $\mathcal{T}$ be a Delaunay triangulation of 
$\mathcal{T}_{0}.$ Then all side lengths of $\mathcal{T}$ are $\leq 2\eta ,$
and any triangle of $\mathcal{T}$ whose side lengths are $\geq \eta $ has
angles $\geq \frac{\pi }{6}.$
\end{theorem}

We begin the proof with

\begin{proposition}
\label{small circumdisk prop}Let $\mathcal{T}_{0}$ be a subset of $\mathbb{R}%
^{2}$ that is $\eta $--dense, and let $\mathcal{T}$ be a Delaunay
triangulation of $\mathcal{T}_{0}.$ If $\Delta $ is a triangle in $\mathcal{T%
}$, and $D_{\Delta }$ is its circumdisk, then the radius of $D_{\Delta }$ is 
$\leq \eta $.
\end{proposition}

\begin{proof}
Let $v$ be a vertex of $\Delta $. Let $r$ be the radius of $D_{\Delta },$
and assume that $r>\eta $. Let $x$ be the center of $D_{\Delta }.$ Since $%
\mathcal{T}$ is Delaunay, it has no vertices in the interior of $D_{\Delta
}. $ Hence the distance from $x$ to all vertices of $\mathcal{T}$ is $>\eta
. $ This contradicts the hypothesis that the vertices of $\mathcal{T}$ are $%
\eta $--dense$.$ So the radius $r$ of $D_{\Delta }$ is $\leq \eta $, as
claimed.
\end{proof}

We are ready to prove the first statement of Theorem \ref{jacked up Chew},
which is the content of the next result.

\begin{proposition}
Let $\mathcal{T}_{0}$ be a subset of $\mathbb{R}^{2}$ that is $\eta $%
--dense, and let $\mathcal{T}$ be a Delaunay triangulation of $\mathcal{T}%
_{0}.$ Then all side lengths of $\mathcal{T}$ are $\leq 2\eta .$

In particular, if $\mathcal{T}$ is a maximal $\eta $--CDG, then all edges of 
$\mathcal{T}$ have length in the interval $[\eta ,2\eta ]$.
\end{proposition}

\begin{proof}
Assume $e$ is an edge of $\Delta $. Since $D_{\Delta }$ has radius $\leq
\eta $, and $e$ is a chord of $D_{\Delta }$, by the triangle inequality, $%
|e|\leq 2\eta $.

If $\mathcal{T}$ is a maximal $\eta $--CDG, then the vertices of $\mathcal{T}
$ are $\eta $--separated, we also have $\eta \leq |e|$.
\end{proof}

To complete the proof of Theorem \ref{jacked up Chew}, we show

\begin{proposition}
Let $\mathcal{T}_{0}$ be a subset of $\mathbb{R}^{2}$ that is $\eta $%
--dense, and let $\mathcal{T}$ be a Delaunay triangulation of $\mathcal{T}%
_{0}.$ If $\Delta =\Delta abc$ is a triangle in $\mathcal{T}$ whose edge
lengths are in $[\eta ,2\eta ]$, then all angles of $\Delta $ are $\geq \pi
/6.$
\end{proposition}

\begin{proof}
Let $D_{\Delta }$ be the circumdisk of $\Delta $ and let $x$ be the center
of $D_{\Delta }$. Since the edges of $\Delta $ are in the interval $[\eta
,2\eta ]$ and the radius $r$ of $D_{\Delta }$ is $\leq \eta $, for any two
vertices, say $a,b$, of $\Delta $, we have 
\begin{equation*}
\sphericalangle axb\geq \frac{\pi }{3}.
\end{equation*}%
By Euclid's Central Angle Theorem, $\sphericalangle acb\geq \pi /6$.
\end{proof}

Before leaving the topic of Chew's Angle Theorem, we record the following
two results on the geometry of the triangles of $\mathrm{CDGs}$ that we use
in the sequel.

\begin{lemma}
\label{dist to opp segm lemma}Let $\Delta =\Delta \left( a,b,c\right) $ be a
triangle with all side lengths $\geq \eta $ and all angles $\geq \frac{\pi }{%
6}.$ If $l_{ab}$ denotes the line through $a$ and $b,$ then 
\begin{equation*}
\mathrm{dist}\left( c,l_{ab}\right) \geq \frac{\eta }{2}.
\end{equation*}%
If equality occurs, then $\sphericalangle \left( a,c,b\right) =\frac{2\pi }{3%
}$ and $\left\vert ca\right\vert =\left\vert cb\right\vert =\eta .$
\end{lemma}

\begin{proof}
At least one of $\sphericalangle a$ or $\sphericalangle b$ is acute. If for
instance $\sphericalangle a$ is acute, then 
\begin{eqnarray*}
\mathrm{dist}\left( c,l_{ab}\right) &=&\left\vert ac\right\vert \sin
\sphericalangle a \\
&\geq &\frac{\eta }{2},
\end{eqnarray*}%
as claimed. Notice that equality forces $\left\vert ac\right\vert =\eta $
and $\sphericalangle a=\frac{\pi }{6},$ and repeating this argument with $%
\sphericalangle b$ shows that in the equality case, $\left\vert
bc\right\vert =\eta $ and $\sphericalangle b=\frac{\pi }{6}.$
\end{proof}

The following is an immediate corollary.

\begin{corollary}
\label{star cor}Let $\mathcal{T}$ be a simplicial complex in $\mathbb{R}^{2}$
with side lengths in $\left[ \eta ,2\eta \right) $ and angles $\geq \frac{%
\pi }{6}.$ Then for any $v\in \mathcal{T}_{0},$%
\begin{equation*}
B\left( v,\frac{\eta }{2}\right) \cap \left\vert \mathcal{T}\right\vert
\subset \bigcup\limits_{\left\{ \left. \mathfrak{s\in }\mathcal{T}\text{ }%
\right\vert \text{ }v\in \left\vert \mathfrak{s}\right\vert \right\}
}\left\vert \mathfrak{s}\right\vert ,
\end{equation*}%
where 
\begin{equation*}
B\left( v,\frac{\eta }{2}\right) :=\left\{ \left. x\in \mathbb{R}^{2}\text{ }%
\right\vert \text{ }\mathrm{dist}\left( x,v\right) <\frac{\eta }{2}\right\} .
\end{equation*}
\end{corollary}

\section{Almost Legal Subdivisions\label{alm legal sub div sect}}

In this section, we prove Theorem \ref{alm leg subdiv thm}. The strategy is
to subdivide the $1$--skeleton and then to extend to the interior of the
original $2$--skeleton. Subdividing the $1$--skeleton turns out to be the
bigger challenge, so for most of the section, we focus on subdividing the
following type of graph.

\begin{definition}
We call a graph \textbf{$\eta $--geometric} provided its vertices are $\eta $%
--separated and its edge lengths are in the interval $\left[ \eta ,2\eta %
\right] .$
\end{definition}

To force our subdivided edges to be part of the new, finer Delaunay
triangulation, in most cases, we arrange that they be in the Delaunay graph.
To do this, we start by recalling

\begin{lemma}
\label{Graph Lemma}\emph{(see Lemma 6.16 on page 74 of \cite{GaHoff})} For $%
S\subset \mathbb{R}^{2},$ a segment $e$ between two points $a,b\in S$ is in
the Delaunay graph of $S$ if and only if the closed disk $D_{e}$ with
diameter $e$ contains no point of $S\setminus \left\{ a,b\right\} .$
\end{lemma}

Using this we will show

\begin{proposition}
\label{Guarn prop}For $\eta >0,$ let $S$ be an $\eta $--Gromov subset of $%
\mathbb{R}^{2}$. If $a,b\in S$ satisfy dist$(a,b)<\sqrt{2}\eta ,$ then the
segment $ab$ between $a$ and $b$ is in the Delaunay graph of $S.$
\end{proposition}

\begin{proof}
By Lemma \ref{Graph Lemma}, $ab$ is in the Delaunay graph of $S$ if the
closed disk $D_{ab}$ with diameter $e$ contains no point of $S\setminus
\left\{ a,b\right\} .$ If dist$(a,b)<\sqrt{2}\eta ,$ then no point in $%
D_{ab} $ is further than $\eta $ from $\left\{ a,b\right\} .$ Since $S$ is
an $\eta $--Gromov set, no point in $D_{ab}$ can be in $S\setminus \left\{
a,b\right\} .$
\end{proof}

Motivated by Proposition \ref{Guarn prop} and Definition \ref{stable edge
dfn}, we make the following definition, wherein the constant $\frac{1}{10}$
could be any small, fixed positive number.

\begin{definition}
\label{legal dfn}We call an edge of an $\eta $--geometric graph $\eta $--%
\textbf{legal }if its length is strictly less than $\left( \sqrt{2}-\frac{1}{%
10}\right) \eta .$ An $\eta $--geometric graph is called \textbf{legal }if
and only if all of its edges are $\eta $--legal.
\end{definition}

\begin{proposition}
\label{edgesub}If $e$ is a line segment with 
\begin{equation*}
\eta \leq \text{length(}e\text{)}\leq 2\eta ,
\end{equation*}%
then there is a subdivision of $e$ into a legal $\frac{\eta }{10}$%
--geometric graph.
\end{proposition}

\begin{proof}
We will subdivide $e$ into subedges of equal length. The number of subedges
will be between $10$ and $19$ and is the following function of the length of 
$e.$ Let 
\begin{equation*}
n_{\eta }:[\eta ,2\eta ]\rightarrow \{10,11,\ldots ,19\}
\end{equation*}%
be the step function that takes the $10$ disjoint subintervals 
\begin{equation*}
\left[ \eta ,\frac{11}{10}\eta \right) ,\left[ \frac{11}{10}\eta ,\frac{12}{%
10}\eta \right) ,\ldots ,\left[ \frac{19}{10}\eta ,2\eta \right]
\end{equation*}%
successively to $\{10,11,\ldots ,19\}.$ Thus for $k=0,1,\ldots ,9,$ 
\begin{equation*}
n_{\eta }(s)=\left\{ 
\begin{array}{ll}
10+k & \text{ if }s\in \lbrack \frac{10+k}{10}\eta ,\frac{11+k}{10}\eta ) \\ 
19 & \text{ if }s=2\eta .%
\end{array}%
\right. .
\end{equation*}

Now subdivide the edge $e$ into $n_{\eta }(\text{length}(e))$ subedges of
equal length. The length of each subedge is then 
\begin{equation*}
f(\text{length}(e)):=\frac{\text{length}(e)}{n_{\eta }(\text{length}(e))}.
\end{equation*}

Notice that the restriction of $f$ to each interval $[\frac{10+k}{10}\eta ,%
\frac{11+k}{10}\eta )$ is increasing. At the endpoints we have 
\begin{equation*}
f\left( \frac{(10+k)\eta }{10}\right) =\frac{\eta }{10}\text{ and }%
\lim_{s\rightarrow \left( \frac{(11+k)\eta }{10}\right) ^{-}}f(s)=\frac{%
(11+k)\eta }{(10+k)10}\leq \frac{11\eta }{10}.
\end{equation*}%
Therefore, for all $s\in \lbrack \eta ,2\eta ]$, 
\begin{equation*}
\frac{\eta }{10}\leq f(s)\leq \frac{11\eta }{10}<\left( \sqrt{2}-\frac{1}{10}%
\right) \frac{\eta }{10},
\end{equation*}%
and the resulting subdivision is a legal $\frac{\eta }{10}$--geometric graph.
\end{proof}

For graphs in $\mathbb{R}$ this proposition gives us

\begin{lemma}
\label{legal subdvi cor}Let $\mathcal{T}$ be an $\eta $--geometric graph in $%
\mathbb{R}.$ There is a subdivision $\mathcal{\tilde{T}}$ of $\mathcal{T}$
that is legal and $\frac{\eta }{10}$--geometric.
\end{lemma}

If instead we assume that $\mathcal{T}$ is an $\eta $--geometric graph in $%
\mathbb{R}^{2}$, applying Proposition \ref{edgesub} to each edge does not in
general lead to an $\frac{\eta }{10}$--separated configuration. The problem
arises when the angle between adjacent edges is small. Even Chew's estimate
that all angles of $\mathrm{CDGs}$ are $\geq \frac{\pi }{6},$ is
insufficient to resolve this issue. So to prove Theorem \ref{alm leg subdiv
thm} we use the following slight modification of the concept of legal
subdivisions.

\begin{definition}
\label{alm legal dfn}Let $\mathcal{T}$ be an $\eta $--geometric graph in $%
\mathbb{R}^{2}.$ A subdivision $\mathcal{\tilde{T}}$ of $\mathcal{T}$ is
called $\frac{\eta }{10}$--\textbf{almost legal} provided $\mathcal{\tilde{T}%
}$ is $\frac{\eta }{10}$--geometric and all edges of $\mathcal{\tilde{T}}$
are $\frac{\eta }{10}$--legal except possibly for edges with a bounding
vertex in $\mathcal{T}_{0}.$ We further require that each edge $\tilde{e}$
with a bounding vertex in $\mathcal{T}_{0}$ satisfies 
\begin{equation*}
\mathrm{length}\left( \tilde{e}\right) =\frac{1}{10}\mathrm{length}\left(
e\right) ,
\end{equation*}%
where $e$ is the edge of $\mathcal{T}$ that contains $\tilde{e}.$
\end{definition}

To prove Theorem \ref{alm leg subdiv thm} it suffices to consider the case
when $\mathcal{T}$ is a maximal $\eta $--\textrm{CDG. }Moreover, it follows
from Proposition \ref{Guarn prop} that a legal edge of an $\frac{\eta }{10}$%
--\textrm{CDG }is $\left( \frac{1}{100}\frac{\eta }{10}\right) $--stable.
Thus Theorem \ref{alm leg subdiv thm} follows from

\begin{theorem}
\label{new alm legal subdiv thm}Every maximal $\eta $--\textrm{CDG }has a
subdivision that is an almost legal $\frac{\eta }{10}$--\textrm{CDG}.
\end{theorem}

To begin the proof, we will subdivide our $1$--skeleton so that it is $\frac{%
\eta }{10}$--almost legal. The following result asserts that this is
possible.

\begin{lemma}
\label{alm legal subdiv lemma}Let $\mathcal{T}$ be an $\eta $--geometric
graph in $\mathbb{R}^{2}.$ There is a subdivision $\mathcal{\hat{T}}$ of $%
\mathcal{T}$ which is $\frac{\eta }{10}$-almost legal.
\end{lemma}

\begin{proof}
For each edge $e$ in $\mathcal{T}$, divide $e$ into 3 subedges so that the
two subedges that contain a vertex of $e$ have length $\text{length}(e)/10$.
With a modification of the step function $n_{\eta }$ used in Proposition \ref%
{edgesub}, produce a legal subdivision of the remaining interior edge.
\end{proof}

\begin{proof}[Proof of Theorem \protect\ref{new alm legal subdiv thm}]
By Lemma \ref{alm legal subdiv lemma} there is an almost legal subdivision $%
\mathcal{G}$ of the $1$--skeleton $\mathcal{T}_{0}\cup \mathcal{T}_{1}$ of $%
\mathcal{T}.$ Let $\mathcal{\tilde{T}}_{0}$ be an extension of $\mathcal{G}%
_{0}$ to an $\frac{\eta }{10}$--Gromov subset of $\mathbb{R}^{2},$ and let $%
\mathcal{\tilde{T}}$ be a Delaunay triangulation of $\mathcal{\tilde{T}}%
_{0}. $ We claim that $\mathcal{\tilde{T}}$ can be chosen to be a
subdivision of $\mathcal{T}.$ This is equivalent to asserting that $\mathcal{%
\tilde{T}}$ can be chosen so that its collection of edges includes the edges
of $\mathcal{G}, $ that is, $\mathcal{G}_{1}\subset \mathcal{\tilde{T}}_{1}.$
To see this for $\tilde{e}\in \mathcal{G}_{1}$ we let $e$ be the unique edge
of $\mathcal{T}$ so that $\tilde{e}\subset e.$ By construction, if $\tilde{e}
$ does not contain a vertex of $e,$ then $\tilde{e}$ is legal and therefore $%
\tilde{e}\in \mathcal{\tilde{T}}_{1}.$ If on the other hand, $\tilde{e}$
does contains a vertex $v$ of $e$, then it is precisely $\frac{1}{10}$ the
length of $e.$ It follows that $\tilde{e}$ is contained in two triangles $%
\tilde{\Delta}_{1},$ $\tilde{\Delta}_{2}$ that are similar to the two
triangles $\Delta _{1},$ $\Delta _{2}$ of $\mathcal{T}$ that contain $e.$
Since $e$ is an edge of the Delaunay triangulation $\mathcal{T},$ it follows
from the angle flip test that we may choose $\mathcal{\tilde{T}}$ so that $%
\tilde{e}\in \mathcal{\tilde{T}}_{1}.$ Thus $\mathcal{\tilde{T}}$ is an $%
\frac{\eta }{10}$--\textrm{CDG }that is an almost legal subdivision of $%
\mathcal{T}.$
\end{proof}

\section{Deforming CDGs\label{epilog section}}

In essence, the proof of Theorem \ref{finiteness} exploits Theorems \ref{CDG
them}, \ref{Gromov-Del Prop}, and \ref{alm leg subdiv thm} together with the
principle that Riemannian manifolds are infinitesimally euclidean. Since
Theorem \ref{finiteness} deals with infinitely many Riemannian manifolds
simultaneously, turning this principle into a rigorous proof requires some
careful analytic arguments on how these results deform. The purpose of this
section is to carry out this analysis.

One issue is that the boundary of an $\eta $-CDG need not be stable in the
sense of Definition \ref{stable edge dfn}. To remedy this, in Subsection \ref%
{stable graph subsect}, we show that every $\eta $-CDG, $\mathcal{T},$ has
extension $\mathcal{T}_{\mathrm{st}}$ whose boundary is stable and is also
not too far from $\mathcal{T}.$ More precisely, $\left\vert \mathcal{T}_{%
\mathrm{st}}\right\vert \subset B\left( \left\vert \mathcal{T}\right\vert
,2\eta \right) $. In Subsection \ref{almost CDGs subsect}, we define a
deformation of the concept of a $\mathrm{CDG}$ that we call an almost $%
\mathrm{CDG.}$ Most of the key properties of almost $\mathrm{CDGs}$ are
proven in Subsection \ref{almost CDGs subsect}. In particular we generalize
Theorem \ref{Gromov-Del Prop}. In Subsection $\ref{alm leg subsect},$ we
complete this process by explaining how to subdivide almost $\mathrm{CDGs}$
in a manner that generalizes Theorem \ref{alm leg subdiv thm}.

\addtocounter{algorithm}{1}

\subsection{The $\protect\xi $--stable Graph\label{stable graph subsect}}

In this subsection, we show we can always choose an extension $\tilde{%
\mathcal{T}}$ of an $\eta $-CDG $\mathcal{T}$ so that the boundary edges of $%
\mathcal{T}_{\mathrm{st}}$ are stable in the sense of Definition \ref{stable
edge dfn}. More specifically, we will prove

\begin{proposition}
\label{stable ext prop}There is an $\xi >0$ with the following property. Let 
$\mathcal{T}$ be an $\eta $--\textrm{CDG }which is a subcomplex of the
maximal $\eta $--\textrm{CDG, }$\mathcal{T}_{\mathrm{\max }}.$ There is an $%
\eta $--\textrm{CDG,} $\mathcal{T}_{\mathrm{st}},$ so that 
\begin{eqnarray*}
\mathcal{T} &\subset &\mathcal{T}_{\mathrm{st}}\subset \mathcal{T}_{\mathrm{%
\max }}, \\
\left\vert \mathcal{T}_{\mathrm{st}}\right\vert &\subset &B\left( \left\vert 
\mathcal{T}\right\vert ,2\eta \right) ,
\end{eqnarray*}%
and the boundary of $\mathcal{T}_{\mathrm{st}}$ consists of $\xi \eta $%
--stable edges.
\end{proposition}

First notice that Proposition \ref{Guarn prop} gives us

\begin{corollary}
\label{stable by len cor}Let $\mathcal{T}$ be an $\eta $--CDG. For all
sufficiently small $\varepsilon >0,$ an edge $e$ of $\mathcal{T}$ is $%
\varepsilon \eta $-stable provided 
\begin{equation*}
\mathrm{length}\left( e\right) <\eta \left( \sqrt{2}-10\varepsilon \right) .
\end{equation*}
\end{corollary}

Next we note that definition of Delaunay triangulation gives us

\begin{corollary}
Let $\mathcal{T}_{0}$ be a discrete point set in $\mathbb{R}^{2}$. For $%
b,d\in \mathcal{T}_{0}$ suppose that $bd$ is an edge of a Delaunay
triangulation of $\mathcal{T}_{0}.$ Then $bd$ is not in the Delaunay graph
of $\mathcal{T}_{0}$ if and only if $bd$ is a diagonal of a quadrilateral $Q$
of $\mathcal{T}$ that is inscribed in a disk $D$ and $\text{int }D\cap 
\mathcal{T}_{0}=\emptyset $. \emph{(Cf. Corollary 6.17 in \cite{GaHoff}.)}
\end{corollary}

Combining this with the definition of $\xi $-stable, gives us the following
result, wherein we use the term $\xi $-unstable to denote an edge $e$ that
is in some Delaunay triangulation of $\mathcal{T}_{0}$ but is not $\xi $%
--stable.

\begin{corollary}
There is a $\xi >0$ with the following property. Let $\mathcal{T}$ be a
maximal $\eta $--CDG and let $bd\in \mathcal{T}_{1}$ be the diagonal of the
quadrilateral $Q$ of $\mathcal{T}.$ Then $bd\in \mathcal{T}_{1}$ is $\eta
\xi $--unstable if and only if there is a disk $D$ that is related to the
vertex set $Q_{0}$ of $Q$ as follows: 
\begin{equation}
Q_{0}=D\cap \mathcal{T}_{0}\text{ and }Q_{0}\subset B(\partial D,\eta \xi ),
\label{q almo inscr}
\end{equation}%
where $B(\partial D,\eta \xi )$ denotes the open $\eta \xi $--ball around $%
\partial D.$
\end{corollary}

It is possible that the boundary edges of $Q$ are unstable, but the manner
in which this can happen is constrained by the rigidity of $\mathrm{CDGs}$
via the following result.

\begin{proposition}
\label{unstableflip}For every $\zeta >0$ there is a $\xi >0$ with the
following property. Let $\mathcal{T}$ be a maximal $\eta $--CDG. If an edge $%
bd\in \mathcal{T}_{1}$ is $\eta \xi $--unstable, then all of the following
hold.

\begin{enumerate}
\item There is an $n$--gon $\mathcal{P}\ $and a disk $D$ so that the vertex
set $\mathcal{P}_{0}$ of $\mathcal{P}$ is related to the disk $D$ as
follows: 
\begin{equation}
\left\{ b,d\right\} \subset \mathcal{P}_{0}=D\cap \mathcal{T}_{0},\text{ and 
}\mathcal{P}_{0}\subset B(\partial D,\eta \zeta ).  \label{alm inscrbd}
\end{equation}%
Moreover, $n=4,5,$ or $6.$

\item The vertices $b,d$ are non-adjacent vertices of $\mathcal{P}$.

\item The boundary edges of $\mathcal{P}$ are $\eta \xi $-stable.
\end{enumerate}
\end{proposition}

\begin{proof}
From the previous result, we have that there is $\xi >0$ so that if $bd\in 
\mathcal{T}_{1}$ is $\eta \xi $--unstable, then there is a disk $D$ and a
quadrilateral $Q:=Q\left( a,b,c,d\right) $ that satisfy Conclusions $1$ and $%
2$ with $\zeta =\xi $. If the boundary edges of $Q$ are $\eta \xi $-stable,
we are done. Otherwise, one of the boundary edges, say $ab$ of $Q$ is $\eta
\xi $-unstable. By the previous corollary, $ab$ is the diagonal of a
quadrilateral $\tilde{Q}$ whose vertex set contains $a,b,d,$ and one
additional point $p.$ Moreover, there is a disk $\tilde{D}$ that is related
to $\tilde{Q}$ as in (\ref{q almo inscr}). Since 
\begin{equation*}
\left\{ a,b,d\right\} \subset B(\partial D,\eta \xi )\text{ and }\left\{
a,b,d\right\} \subset B(\partial \tilde{D},\eta \xi ),
\end{equation*}%
\begin{equation*}
\mathrm{dist}_{\mathrm{Haus}}\left( D,\tilde{D}\right) <\tau \left( \eta \xi
\right) ,
\end{equation*}%
where $\tau :\mathbb{R}\rightarrow \mathbb{R}_{+}$ is some function that
satisfies $\lim_{t\rightarrow 0}\tau \left( t\right) =0.$

Thus the pentagon $\mathcal{P}\left( a,p,b,c,d\right) $ and the disk $D$
satisfy (\ref{alm inscrbd}) with $\zeta :=\tau \left( \xi \right) .$ Again
we are done if the boundary edges of $\mathcal{P}$ are $\eta \xi $-stable.

Otherwise, we repeat the argument above and obtain a hexagon $\mathcal{H}$
of $\mathcal{T}$ whose vertex set includes $\left\{ a,p,b,c,d\right\} $ and
satisfies (\ref{alm inscrbd}) for some disk $\hat{D}.$ In an $\eta $--%
\textrm{CDG }all edge lengths are $\geq \eta $ and all circumdisks have
radius $<\eta ;$ so if an $n$--gon $\mathcal{P}$ of an $\eta $--\textrm{CDG }%
satisfies (\ref{alm inscrbd}) for some disk $D,$ then $n\leq 6,$ and all
edge lengths of $\mathcal{P}$ are nearly $\eta .$ Thus by Corollary \ref%
{stable by len cor}, there is an $\xi _{6}>0$ so that all edges of $\mathcal{%
P}$ are $\xi _{6}\eta $--stable.
\end{proof}

\begin{proof}[Proof of Proposition \protect\ref{stable ext prop}]
For $\zeta \in \left( 0,\frac{1}{1,000}\right) ,$ let $\xi >0$ be as in
Proposition \ref{unstableflip}. Let $e$ be an $\eta \xi $--unstable boundary
edge of $\mathcal{T}.{\huge \ }$By Proposition \ref{unstableflip}, there is
an $n\in \left\{ 4,5,6\right\} $ so that $e$ connects two nonadjacent
vertices of an $n$--gon $\mathcal{P}$ of $\mathcal{T}_{\mathrm{\max }},$ $%
\mathcal{P}$ is almost inscribed in a disk, and the boundary edges of $%
\mathcal{P}$ are $\eta \xi $-stable.

Our edge $e$ separates $\left\vert \mathcal{P}\right\vert $ into two
components. Call them $C_{-}$ and $C_{+},$ and let $C_{-}$ be the component
that contains the face of $\mathcal{T}$ that has $e$ in its boundary.

Now form a new complex $\mathcal{T}_{e}$ which is the union of $\mathcal{T}$
with the simplices of $\mathcal{T}_{\mathrm{\max }}\setminus \mathcal{T}$
that are contained in $\left\vert \mathcal{P}\right\vert \cap C_{+}.$ Since $%
\mathcal{P}$ is almost inscribed in a disk of radius $\eta ,\ $and since $%
\mathcal{P}$ is an $n$--gon with $n\in \left\{ 4,5,6\right\} $ and side
lengths $\geq \eta ,$ 
\begin{equation*}
\mathcal{T}_{\mathrm{e}}\setminus \mathcal{T}\subset B\left( \left\vert
e\right\vert ,2\eta \right) .
\end{equation*}%
Thus 
\begin{equation*}
\mathcal{T}_{\mathrm{e}}\subset B\left( \left\vert \mathcal{T}\right\vert
,2\eta \right) .
\end{equation*}%
Since the number of unstable boundary edges of $\mathcal{T}_{\mathrm{e}}$ is
one less than the number of unstable boundary edges of $\mathcal{T},$
repeating this process a finite number of times completes the proof.
\end{proof}

\addtocounter{algorithm}{1}

\subsection{Almost CDGs\label{almost CDGs subsect}}

In this subsection we define almost $\mathrm{CDGs}$ and discuss their key
properties. First, notice that $A\subset X$ is $\eta $--Gromov if and only
if 
\begin{equation*}
B\left( A,\eta \right) =X,
\end{equation*}%
and for distinct $a,b\in A,$ 
\begin{equation*}
\mathrm{dist}\left( a,b\right) \geq \eta .
\end{equation*}%
So the following is a natural deformation of this condition.

\begin{definition}
Given $\eta ,\delta >0,$ we say a subset $A$ of a metric space $X$ is an $%
\left( \eta ,\delta \eta \right) $--\textbf{Gromov set} if and only if 
\begin{equation*}
B\left( A,\eta \left( 1+\delta \right) \right) =X,
\end{equation*}%
and for distinct $a,b\in A,$ 
\begin{equation*}
\mathrm{dist}\left( a,b\right) \geq \eta \left( 1-\delta \right) .
\end{equation*}
\end{definition}

\begin{figure}[tbp]
\includegraphics[scale=.20]{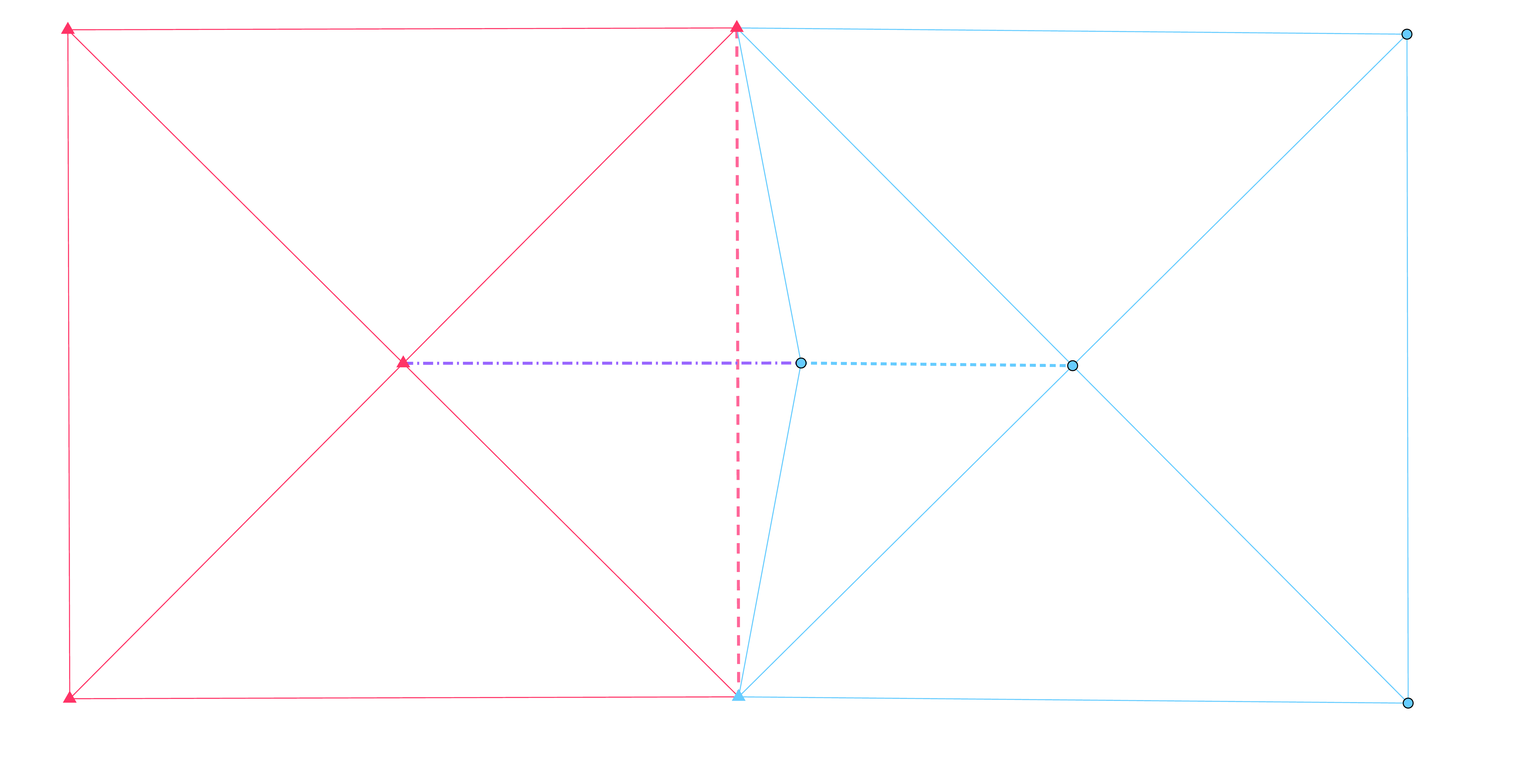}
\caption{{\protect\tiny {Assume the side lengths of these squares is just
under $2$. Then the red vertices are a $1$--Gromov subset of the red square,
and the red edges are a Delaunay triangulation of the red vertices. The blue
vertices are an extension of the red ones to a $1$--Gromov subset of the
union of the two squares, but the dashed red edge fails the angle flip test
and hence is not Delaunay for the extended configuration. The Delaunay
triangulation of the extended set consists of all of the edges pictured
except for the dashed red one.}}}
\label{sqrt3}
\end{figure}

A naive definition of $\mathrm{CDGs}$ is that they are Delaunay
triangulations of Gromov sets. The problem with this idea is that while
Gromov sets can always be extended, their Delaunay triangulations cannot
necessarily be extended (see Figure \ref{sqrt3}). Our actual definition of $%
\mathrm{CDGs}$ skirts this issue by placing $\mathrm{CDGs}$ inside of
maximal ones. In our application, almost $\mathrm{CDGs}$ will not be
presented inside of maximal complexes, but fortunately the technicality
presented in Figure \ref{sqrt3} is local at heart, and the following
definition exploits this fact.

\begin{definitionnonum}
\label{ext grm dfn}Let $X$ be a metric space with $U,U_{\mathrm{x}}\subset X$
open in $X.$ Given $\eta >0$ and $\delta \geq 0,$ let $A$ and $A_{\mathrm{x}%
} $ be $\left( \eta ,\delta \eta \right) $--Gromov subsets of the closures
of $U$ and $U_{\mathrm{x}}$ respectively. We say that $A_{\mathrm{x}}$ is a 
\textbf{buffer }of $A$ provided:

\begin{enumerate}
\item The closed ball $D(U,6\eta )\subset U_{\mathrm{x}}$.

\item $A=A_{\mathrm{x}}\cap \mathrm{closure}\left( U\right) .$
\end{enumerate}
\end{definitionnonum}

While a perturbation of a $\mathrm{CDG}$ can fail the angle flip test, it
will nevertheless pass the following easier test.

\begin{definition}
Let $\mathcal{T}$ be a simplicial complex in $\mathbb{R}^{2}.$ Let $e$ be an
edge of $\mathcal{T}$ which bounds two faces. We say that $e$ passes the $%
\varepsilon $--angle flip test if and only if the sum of the angles opposite 
$e$ is $\leq \pi +\varepsilon .$
\end{definition}

We are now ready for the definition of almost $\mathrm{CDGs.}$

\begin{definition}
\label{alm CDG dfn}Given $\eta ,\delta ,\varepsilon >0,$ let $\mathcal{T}%
_{0} $ be an $\left( \eta ,\delta \eta \right) $--Gromov subset of the
closure of an open set $U\subset \mathbb{R}^{2}$ with buffer $\left( 
\mathcal{T}_{0}\right) _{\mathrm{x}}.$ A triangulation $\mathcal{T}$ of $%
\mathcal{T}_{0} $ is called an $\left( \eta ,\delta \eta ,\varepsilon
\right) $--CDG provided $\mathcal{T}$ is a subcomplex of a triangulation $%
\mathcal{T}_{\mathrm{x}}$ of $\left( \mathcal{T}_{0}\right) _{\mathrm{x}}$,
and each edge of $\mathcal{T}_{\mathrm{x}}$ passes the $\varepsilon $--angle
flip test.
\end{definition}

Arguing as in the proof of Proposition \ref{stable ext prop}, we get

\begin{corollary}
\label{stable edge cor}There are $\varepsilon ,\xi ,\delta >0$ with the
following property. If $\mathcal{T}$ is an $\left( \eta ,\delta \eta
,\varepsilon \right) $--\textrm{CDG }with buffer $\mathcal{T}_{\mathrm{x}},$
then there is an $\left( \eta ,\delta \eta ,\varepsilon \right) $--\textrm{%
CDG,} $\mathcal{T}_{\mathrm{st}},$ so that 
\begin{eqnarray*}
\mathcal{T} &\subset &\mathcal{T}_{\mathrm{st}}\subset \mathcal{T}_{\mathrm{x%
}}, \\
\left\vert \mathcal{T}_{\mathrm{st}}\right\vert &\subset &B\left( \left\vert 
\mathcal{T}\right\vert ,2\eta \right) ,
\end{eqnarray*}%
and the boundary of $\mathcal{T}_{\mathrm{st}}$ consists of $\xi \eta $%
--stable edges.
\end{corollary}

\begin{definition}
We say that an $\left( \eta ,\delta \eta ,\varepsilon \right) $--CDG is $\xi
\eta $--stable provided each of its boundary edges $\mathcal{T}$ is $\xi
\eta $--stable$.$
\end{definition}

While an $\left( \eta ,\delta \eta ,\varepsilon \right) $--CDG is not
necessarily Delaunay, it is almost Delaunay in the following sense.

\begin{definition}
A triangulation $\mathcal{T}$ of a point set of $\mathbb{R}^{2}$ is called $%
\varepsilon $--\textbf{Delaunay} if and only if for each $2$--simplex $%
\Delta \in \mathcal{T},$ all vertices of $\mathcal{T}$ that are in $%
D_{\Delta }$ are within $\varepsilon $ of $\partial D_{\Delta }$.
\end{definition}

The proof of Proposition \ref{two tests are one prop} gives us

\begin{proposition}
\label{equiv flips prop}Given $\kappa >0$ and a sufficiently small $\delta
>0 $, there is an $\varepsilon >0$ so that every $\left( \eta ,\delta \eta
,\varepsilon \right) $-CDG is $\kappa $-Delaunay.
\end{proposition}

Arguing as in the proof of Chew's Angle Theorem (\ref{Gromov-Del Prop}) we
get

\begin{proposition}
\label{almost chew}Given $d,\theta >0$, there are $\varepsilon ,\delta >0$
so that for every $\left( \eta ,\delta \eta ,\varepsilon \right) $-CDG,

\begin{enumerate}
\item The radius of every circumdisk of every two simplex of $\mathcal{T}$
is $\leq \eta +d.$

\item Every edge length of $\mathcal{T}$ is in the interval $\left[ \eta
-d,2\eta +d\right] .$

\item All angles of $\mathcal{T}$ are $\geq \frac{\pi }{6} -\theta.$
\end{enumerate}
\end{proposition}

Using Proposition \ref{equiv flips prop} we get

\begin{proposition}
\label{Dle is local only prop}Given $\kappa >0$, there are $\varepsilon
,\delta >0$ that satisfy the following. Given any $\left( \eta ,\delta \eta
,\varepsilon \right) $--CDG, $\mathcal{T},$ in $\left[ 0,\beta \right]
\times \left[ 0,\beta \right] \subset \mathbb{R}^{2}$ and any finite set of
points $V_{0}$ in $\mathbb{R}^{2}\setminus B\left( \mathcal{T}_{0},3\eta
\right) $, there is a $\kappa $-Delaunay triangulation of $\mathcal{T}%
_{0}\cup V_{0}$ that contains $\mathcal{T}$.
\end{proposition}

\begin{proof}
Choose $\delta ,\varepsilon >0$ so that Part 1 of Proposition \ref{almost
chew} holds with $d=\eta /2$ and so that Proposition \ref{equiv flips prop}
implies $\mathcal{T}$ is $\kappa $-Delaunay. Let $\Delta \in \mathcal{T}_{2}$
have circumdisk $D_{\Delta }.$ Since $\mathcal{T}$ is a $\kappa $-Delaunay
triangulation of $\mathcal{T}_{0},$%
\begin{equation*}
D_{\Delta }\cap \mathcal{T}_{0}\subset B\left( \partial D_{\Delta },\kappa
\right) .
\end{equation*}%
Since $V_{0}\subset \mathbb{R}^{2}\setminus B\left( \mathcal{T}_{0},3\eta
\right) $ and the radius of $D_{\Delta }$ is less than $\frac{3}{2}\eta ,$ 
\begin{equation*}
D_{\Delta }\cap V_{0}=\varnothing .
\end{equation*}%
Thus 
\begin{equation*}
D_{\Delta }\cap \left( \mathcal{T}_{0}\cup V_{0}\right) \subset B\left(
\partial D_{\Delta },\kappa \right) ,
\end{equation*}%
and $\Delta $ is a $\kappa $-Delaunay triangle of $\mathcal{T}_{0}\cup
V_{0}. $
\end{proof}

We can now state the main result of the section, which is the following
extension theorem for almost $\mathrm{CDGs.}$ It has the added feature that
error estimates, $\varepsilon $ and $\delta ,$ for the new simplices can be
chosen to be $0.$

\begin{theorem}
\label{real chew thm}For $\varepsilon $, $\xi ,$ and $\delta $ as in
Corollary \ref{stable edge cor} and $\beta \geq 50\eta ,$ let $\mathcal{T}$
be a $\xi \eta $--stable $\left( \eta ,\delta \eta ,\varepsilon \right) $%
--CDG in $\left[ 0,\beta \right] \times \left[ 0,\beta \right] \subset 
\mathbb{R}^{2}.$ There\ is a $\xi \eta $--stable $\left( \eta ,\delta \eta
,\varepsilon \right) $--CDG, $\mathcal{\tilde{T}},$\ which extends $\mathcal{%
T}$\ and has the following properties.

\begin{enumerate}
\item $\left[ 0,\beta \right] \times \left[ 0,\beta \right] \subset B\left( 
\mathcal{\tilde{T}}_{0},\eta \right) .$

\item The buffer of $\mathcal{\tilde{T}}$ extends the buffer of $\mathcal{T}%
, $ that is, 
\begin{equation}
\left( \mathcal{T}_{\mathrm{x}}\right) _{0}\subset \left( \mathcal{\tilde{T}}%
_{\mathrm{x}}\right) _{0}.  \label{extention}
\end{equation}

\item $\left( \mathcal{\tilde{T}}_{\mathrm{x}}\right) _{0}$ is a subset of $%
\left[ 0,\beta +6\eta \right] \times \left[ 0,\beta +6\eta \right] $ so that%
\begin{equation}
\mathrm{dist}\left( \left( \mathcal{\tilde{T}}_{\mathrm{x}}\right)
_{0}\setminus \left( \mathcal{T}_{\mathrm{x}}\right) _{0},\mathcal{T}%
_{0}\right) \geq 3\eta ,  \label{only far away}
\end{equation}%
\begin{equation}
\left[ 0,\beta +6\eta \right] \times \left[ 0,\beta +6\eta \right] \subset
B\left( \left( \mathcal{\tilde{T}}_{\mathrm{x}}\right) _{0},\eta \right) ,%
\text{ and\label{eta dense}}
\end{equation}%
\begin{equation}
\mathrm{dist}\left( v,w\right) \geq \eta ,  \label{eta sep}
\end{equation}

for $v\in \left( \mathcal{\tilde{T}}_{\mathrm{x}}\right) _{0}\setminus
\left( \mathcal{T}_{\mathrm{x}}\right) _{0}$ and $w\in \left( \mathcal{%
\tilde{T}}_{\mathrm{x}}\right) _{0}.$

\item $\mathcal{\tilde{T}}_{\mathrm{x}}$ contains all legal subgraphs of $%
\mathcal{T}_{\mathrm{x}}.$

\item Every edge of $\mathcal{\tilde{T}}$ that has a vertex in $\left( 
\mathcal{\tilde{T}}_{0}\right) _{\mathrm{x}}\setminus \mathcal{T}_{0}$ is
Delaunay in the sense that it passes the $\varepsilon $--angle flip test
with $\varepsilon =0$.

\item Every edge with a vertex in $\left( \mathcal{\tilde{T}}_{0}\right) _{%
\mathrm{x}}\setminus \left( \mathcal{T}_{0}\right) _{\mathrm{x}}$ has length 
$\leq 2\eta .$
\end{enumerate}
\end{theorem}

\begin{proof}
We construct $\left( \mathcal{\tilde{T}}_{0}\right) _{\mathrm{x}}$ by
consecutively choosing points in $\left[ 0,\beta +6\eta \right] \times \left[
0,\beta +6\eta \right] \setminus B\left( \mathcal{T}_{0},3\eta \right) $
that are $\eta $--separated from each other and from $\left( \mathcal{T}%
_{0}\right) _{\mathrm{x}}.$ By compactness the process ends in a finite
number of steps. The final set, $\left( \mathcal{\tilde{T}}_{0}\right) _{%
\mathrm{x}}$ satisfies (\ref{extention}), (\ref{only far away}), and (\ref%
{eta sep}) by construction. $\left( \mathcal{\tilde{T}}_{0}\right) _{\mathrm{%
x}}$ also satisfies (\ref{eta dense}), since otherwise the construction
would have continued for at least one more step.

Let $\mathcal{\tilde{T}}_{\mathrm{x}}$ be an $\varepsilon $--Delaunay
triangulation of $\left( \mathcal{\tilde{T}}_{0}\right) _{\mathrm{x}}$ in
the sense that every edge passes the $\varepsilon $--angle flip test. Since 
\begin{equation*}
\left( \mathcal{\tilde{T}}_{0}\right) _{\mathrm{x}}\setminus \left( \mathcal{%
T}_{0}\right) _{\mathrm{x}}\subset \mathbb{R}^{2}\setminus B\left( \mathcal{T%
}_{0},3\eta \right) ,
\end{equation*}%
it follows from Proposition \ref{Dle is local only prop} that we may choose
the triangulation of $\mathcal{\tilde{T}}$ to be one that extends $\mathcal{T%
}.$ By flipping the edges of $\mathcal{\tilde{T}}_{\mathrm{x}}$ that are not
Delaunay and have a vertex in $\left( \mathcal{\tilde{T}}_{0}\right) _{%
\mathrm{x}}\setminus \mathcal{T}_{0},$ we force $\mathcal{\tilde{T}}_{%
\mathrm{x}}$ to satisfy Conclusion 5. Since the boundary edges of $\mathcal{T%
}$ are $\xi \eta $--stable, we can do this without needing to flip these
boundary edges and while preserving the condition that $\mathcal{\tilde{T}}_{%
\mathrm{x}}$ is an $\varepsilon $--Delaunay triangulation of $\left( 
\mathcal{\tilde{T}}_{0}\right) _{\mathrm{x}}.$ Since an edge passes the
angle flip test if and only if it passes the Lawson flip test, an edge with
a vertex in $\left( \mathcal{\tilde{T}}_{0}\right) _{\mathrm{x}}\setminus
\left( \mathcal{T}_{0}\right) _{\mathrm{x}}$ is the diameter of a disk that
has no vertices in its interior. Part 6 follows from this and (\ref{eta
dense}).

Let 
\begin{equation*}
\mathcal{\tilde{T}}_{0}^{\mathrm{pre}}:=\left( \mathcal{\tilde{T}}%
_{0}\right) _{\mathrm{x}}\cap \left\{ \left[ 0,\beta +\eta \right] \times %
\left[ 0,\beta +\eta \right] \right\} .
\end{equation*}

Let $\mathcal{\tilde{T}}^{\mathrm{pre}}$ be the subcomplex of $\mathcal{%
\tilde{T}}_{\mathrm{x}}$ all of whose vertices are in $\mathcal{\tilde{T}}%
_{0}^{\mathrm{pre}}.$ By Corollary \ref{stable edge cor}, there is a $\xi
\eta $--stable $\left( \eta ,\delta \eta ,\varepsilon \right) $--CDG, $%
\mathcal{\tilde{T}},$ so that 
\begin{eqnarray*}
\mathcal{T}^{\mathrm{pre}} &\subset &\mathcal{\tilde{T}}\subset \mathcal{%
\tilde{T}}_{\mathrm{x}}\text{ and} \\
\left\vert \mathcal{\tilde{T}}\right\vert &\subset &B\left( \left\vert 
\mathcal{\tilde{T}}^{\mathrm{pre}}\right\vert ,2\eta \right) .
\end{eqnarray*}

It follows from Proposition \ref{Guarn prop} and the definition of a legal
graph (\ref{legal dfn}) that $\mathcal{\tilde{T}}$ contains all legal
subgraphs of $\mathcal{T}_{\mathrm{x}}.$
\end{proof}

Informally, Parts 3 and 5 of the previous result say that $\mathcal{\tilde{T}%
}$ is an extension of $\mathcal{T}$ that corrects the two error estimates, $%
\delta $ and $\varepsilon .$ Motivated by this, we propose

\begin{definition}
If $\mathcal{T}$ and $\mathcal{\tilde{T}}$ are related as in Theorem \ref%
{real chew thm}, then we will say that $\mathcal{\tilde{T}}$ is an \textbf{%
error correcting extension }of $\mathcal{T}.$
\end{definition}

\addtocounter{algorithm}{1}

\subsection{Almost Legal Subdivisions of Almost \textrm{CDGs\label{alm leg
subsect}}}

In this subsection, we show how the proof of Theorem \ref{new alm legal
subdiv thm} gives us almost legal subdivisions of almost $\mathrm{CDGs.}$

\begin{definition}
Given $\eta ,\delta >0,$ we say that a graph $\mathcal{T}$ is $\left( \eta
,\delta \eta \right) $--geometric provided its vertices are $\eta \left(
1-\delta \right) $--separated and its edge lengths are all in $\left[ \eta
\left( 1-\delta \right) ,2\eta \left( 1-\delta \right) \right] .$
\end{definition}

\begin{definition}
Let $\mathcal{T}$ be an $\left( \eta ,\delta \eta \right) $--geometric graph$%
.$ Given $\eta ,\delta >0,$ a subdivision $\mathcal{\tilde{T}}$ of $\mathcal{%
T}$ is called $\left( \frac{\eta }{10},\delta \frac{\eta }{10}\right) $--%
\textbf{almost legal} if and only if $\mathcal{\tilde{T}}_{0}$ is $\left( 
\frac{\eta }{10}\left( 1-\delta \right) \right) $--separated and all edges
of $\mathcal{\tilde{T}}$ have length in the interval $\left[ \frac{\eta }{10}%
,\frac{\eta }{10}\left( \sqrt{2}-\frac{1}{10}\right) \right] ,$ except
possibly for edges with a bounding vertex in $\mathcal{T}_{0}.$ We further
require that each edge $\tilde{e}$ with a bounding vertex in $\mathcal{T}%
_{0} $ satisfies 
\begin{equation*}
\mathrm{length}\left( \tilde{e}\right) =\frac{1}{10}\mathrm{length}\left(
e\right) ,
\end{equation*}%
where $e$ is the edge of $\mathcal{T}$ that contains $\tilde{e}.$
\end{definition}

With a minor numerical adjustment, which we leave to the reader, the proof
of Lemma \ref{alm legal subdiv lemma} gives us

\begin{corollary}
Let $\mathcal{T}$ be an $\left( \eta ,\delta \eta \right) $--geometric graph
in $\mathbb{R}^{2}.$ If $\delta $ is sufficiently small, then there is a
subdivision $\mathcal{\hat{T}}$ of $\mathcal{T}$ which is $\left( \frac{\eta 
}{10},\delta \frac{\eta }{10}\right) $-almost legal.{\large \ }
\end{corollary}

Let $\mathcal{T}$ be an $\left( \eta ,\delta \eta ,\varepsilon \right) $%
--CDG. Let $\mathcal{L}\left( \text{sk}_{1}\left( \mathcal{T}_{\mathrm{x}%
}\right) \right) $ be the almost legal subdivision of the 1-skeleton sk$_{1}(%
\mathcal{T})$ obtained by applying the previous corollary to sk$_{1}\left( 
\mathcal{T}_{\mathrm{x}}\right) .$ As in the proof of Theorem \ref{new alm
legal subdiv thm}, extend $\mathcal{L}\left( \text{sk}_{1}\left( \mathcal{T}%
_{\mathrm{x}}\right) \right) _{0}$ to a maximal subset $\mathcal{L}\left( 
\mathcal{T}_{\mathrm{x}}\right) _{0}$ of $\left\vert \mathcal{T}_{\mathrm{x}%
}\right\vert $ so that 
\begin{equation*}
\mathrm{dist}\left( v,w\right) \geq \frac{\eta }{10}
\end{equation*}%
for all $v\in \mathcal{L}\left( \mathcal{T}_{\mathrm{x}}\right) _{0}$ and
all $w\in \mathcal{L}\left( \mathcal{T}_{\mathrm{x}}\right) _{0}\setminus 
\mathcal{L}\left( \text{sk}_{1}\left( \mathcal{T}_{\mathrm{x}}\right)
\right) _{0}.$ Define $\mathcal{L}\left( \mathcal{T}\right) _{0}$
analogously, and let $\mathcal{L}\left( \mathcal{T}_{\mathrm{x}}\right) $
and $\mathcal{L}\left( \mathcal{T}\right) $ be $\varepsilon $--Delaunay
triangulations of $\mathcal{L}\left( \mathcal{T}_{\mathrm{x}}\right) _{0}$
and $\mathcal{L}\left( \mathcal{T}\right) _{0},$ respectively. Arguing as in
the proof of Theorem \ref{new alm legal subdiv thm}, we see that $\mathcal{L}%
\left( \mathcal{T}\right) $ is an $\left( \frac{\eta }{10},\delta \frac{\eta 
}{10},\varepsilon \right) $--\textrm{CDG }which is a subdivision of $%
\mathcal{T}$.

This construction respects error correcting extensions. In fact,

\begin{theorem}
Let $\mathcal{T}$ be an $\left( \eta ,\delta \eta ,\varepsilon \right) $%
--CDG in $\left[ 0,\beta \right] \times \left[ 0,\beta \right] \subset 
\mathbb{R}^{2}.$ If $\mathcal{\tilde{T}}$ is an error correcting extension
of $\mathcal{T},$ then there are subdivisions $\mathcal{L}\left( \mathcal{T}%
\right) ,$ $\mathcal{L}\left( \mathcal{\tilde{T}}\right) $ of $\mathcal{T}$
and $\mathcal{\tilde{T}}$, respectively, so that

\begin{enumerate}
\item $\mathcal{L}\left( \mathcal{T}\right) $ and $\mathcal{L}\left( 
\mathcal{\tilde{T}}\right) $ are $\left( \frac{\eta }{10},\delta \frac{\eta 
}{10},\varepsilon \right) $--\textrm{CDGs}.

\item $\mathcal{L}\left( \mathcal{\tilde{T}}\right) $ is an error correcting
of $\mathcal{L}\left( \mathcal{T}\right) .$

\item If $\tilde{G}\subset \mathcal{\tilde{T}}_{1}$ is a legal subgraph of
the $1$--skeleton of $\mathcal{\tilde{T}}$ and $\mathcal{L}\left( \tilde{G}%
\right) $ is a legal, $\left( \frac{\eta }{10},\delta \frac{\eta }{10}%
\right) $\textbf{--}geometric\textbf{\ }subdivision of $\tilde{G}$, then we
can choose $\mathcal{L}\left( \mathcal{\tilde{T}}\right) $ so that it
contains $\mathcal{L}\left( \tilde{G}\right) .$
\end{enumerate}
\end{theorem}

\begin{proof}
With one exception, everything follows from the construction. The
exceptional property is that 
\begin{equation*}
\mathrm{dist}\left( v,w\right) \geq \frac{\eta }{10}
\end{equation*}%
for $v\in \mathcal{L}\left( \mathcal{\tilde{T}}_{\mathrm{x}}\right)
_{0}\setminus \mathcal{L}\left( \mathcal{T}_{\mathrm{x}}\right) _{0}$ and $%
w\in \mathcal{L}\left( \mathcal{T}_{\mathrm{x}}\right) _{0}.$ This, also, is
immediate from the construction if either point is in $\mathcal{L}\left( 
\text{sk}_{1}\left( \mathcal{\tilde{T}}_{\mathrm{x}}\right) \right) _{0}$.

So suppose that $v\in \mathcal{L}\left( \mathcal{\tilde{T}}_{\mathrm{x}%
}\right) _{0}\setminus \mathcal{L}\left( \mathcal{T}_{\mathrm{x}}\right)
_{0} $ and $w\in \mathcal{L}\left( \mathcal{T}_{\mathrm{x}}\right) _{0}$
satisfy 
\begin{equation}
\mathrm{dist}\left( v,w\right) <\frac{\eta }{10},  \label{nope}
\end{equation}%
and neither point is in $\mathcal{L}\left( \text{sk}_{1}\left( \mathcal{%
\tilde{T}}_{\mathrm{x}}\right) \right) _{0}.$ It follows that both $v$ and $%
w $ are $\frac{\eta }{10}$--separated from $\mathcal{L}\left( \text{sk}%
_{1}\left( \mathcal{\tilde{T}}_{\mathrm{x}}\right) \right) _{0}.$ Applying
Corollary \ref{star cor} with $\frac{\eta }{10}$ playing the role of $\eta ,$
and using (\ref{nope}), we see that $v$ and $w$ must be opposite vertices of
two triangles that share an edge. By Lemma \ref{dist to opp segm lemma},
this edge must nearly have length $\sqrt{3}\frac{\eta }{10},$ and the angles
at $v$ and $w$ of the respective triangles must nearly be $\frac{2\pi }{3}.$
Since such an edge fails the angle flip test by a large margin, no such
configuration can exist.
\end{proof}

The notion of error correcting extensions also makes sense for graphs in $%
\mathbb{R}.$

\begin{definition}
Let $\mathcal{\tilde{T}}$ and $\mathcal{T}$ be $\left( \eta ,\delta \eta
\right) $--geometric graphs with $\mathcal{T}\subset \mathcal{\tilde{T}}$
and $\left\vert \mathcal{T}\right\vert ,\left\vert \mathcal{\tilde{T}}%
\right\vert \subset \left[ 0,\beta \right] .$ We say that $\mathcal{\tilde{T}%
}$ is an error correcting extension of $\mathcal{T}$ provided $\left( 
\mathcal{\tilde{T}}\right) _{0}$ is a maximal subset of $\left[ 0,\beta %
\right] $ so that 
\begin{equation*}
\mathrm{dist}\left( v,w\right) \geq \eta
\end{equation*}%
for all $v\in \mathcal{\tilde{T}}_{0}\ $and all $w\in \mathcal{T}_{0}.$
\end{definition}

Via simpler arguments we get

\begin{corollary}
\label{perfect and subdiv cor}Let $\mathcal{\tilde{T}}$ and $\mathcal{T}$ be 
$\left( \eta ,\delta \eta \right) $--geometric graphs with $\left\vert 
\mathcal{T}\right\vert ,\left\vert \mathcal{\tilde{T}}\right\vert \subset %
\left[ 0,\beta \right] .$ If $\mathcal{\tilde{T}}$ is an error correcting
extension of $\mathcal{T},$ then we can choose the legal subdivisions $%
\mathcal{L}\left( \mathcal{T}\right) $ and $\mathcal{L}\left( \mathcal{%
\tilde{T}}\right) $ of Lemma \ref{legal subdvi cor} so that $\mathcal{L}%
\left( \mathcal{\tilde{T}}\right) $ is an error correcting extension of $%
\mathcal{L}\left( \mathcal{T}\right) .$
\end{corollary}

Chew's angle theorem combined with the fact that Riemannian manifolds are
infinitesimally euclidean and the definition of CDGs immediately yields

\begin{theorem}
\label{compact surfaces thm}Let $S$ be a compact Riemannian $2$--manifold.
For every $\varepsilon >0$ there is an $\eta _{0}$ so that for all $\eta \in
\left( 0,\eta _{0}\right) $ every $\eta $--Gromov subset of $S$ admits a
triangulation $\mathcal{T}$ for which all side lengths are in $\left[ \eta
\left( 1-\varepsilon \right) ,2\eta \left( 1+\varepsilon \right) \right] \ $%
and all angles are $\geq \frac{\pi }{6}-\varepsilon .$
\end{theorem}

The analog of this result for surfaces that are isometrically embedded in $%
\mathbb{R}^{3}$ is proven by Chew in \cite{Chew2}.

\end{document}